\documentclass{article}

  \usepackage{amsmath}
  \usepackage{paralist}
  \usepackage{graphics}
  \usepackage{epsfig}
  \usepackage[latin1]{inputenc}
  \usepackage[english]{babel}
  \usepackage{amssymb}
  \usepackage{hyperref}
  \usepackage{theorem}
  \usepackage{xypic}

\newtheorem{theorem}{Theorem}
\newtheorem{corollary}{Corollary}
\newtheorem{lemma}{Lemma}
\newtheorem{proposition}{Proposition}

{ \theorembodyfont{\rmfamily}
\newtheorem{definition}{Definition}
\newtheorem{remark}{Remark}
\newtheorem{example}{Example} }

\newcommand{\Der}{\mathrm{Der}}
\newcommand{\Aut}{\mathrm{Aut}}
\newcommand{\Ann}{\mathrm{Ann}}

\newenvironment{proof}
  {\begin{trivlist}
  \item[\textit{\noindent\textsc{ Proof.}}]}
  {{\bf q.e.d.}\end{trivlist}}

\begin{document}

\title{Affine Structures on Jet and Weil Bundles}
\author{David Bl\'azquez-Sanz}

\maketitle

\begin{abstract}
Weil algebra morphism induce natural transformations between Weil bundles. In some
well known cases, a natural transformation is endowed with a canonical structure of affine
bundle. We show that this structure arises only when the Weil algebra morphism is surjective
and its kernel has null square. Moreover, in some cases, this structure of affine bundle is
passed down to Jet spaces. We give a characterization of this fact in algebraic terms. This
algebraic condition also determines an affine structure between the groups of automorphisms of
related Weil algebras. 
\end{abstract}

\begin{center}

Mathematical Subject Classification 58A20, 58A32 \\ \medskip
{\bf Key Words: } \\ \medskip
Jet, Weil Bundles, Affine Structure, Natural Transformations.

\end{center}

\medskip

\section*{Introduction}

The theory of Weil bundles and Jet spaces is developed in order to understand
the geometry of PDE systems. C. Ehresmann formalized contact elements of S. Lie,
introducing the spaces of jets of sections; simultaneously A. Weil showed in \cite{Weil1953}
that the theory of S. Lie could be formalized easily by replacing the spaces of contact elements
by the more formal spaces of \emph{``points proches''}, known as Weil bundles. The general
theory of jet spaces \cite{Munoz2000} recovers the classical spaces of contact elements $J^l_mM$
of S. Lie from the ideas and metodology of A. Weil. 

In the theory of Weil bundles, morphisms $A\to B$ of Weil algebras induce \emph{natural transformations}
\cite{Kolar1993} between Weil bundles. There are well known cases in which these natural transformations
are affine bundles that often appear in differential geometry \cite{Kolar1993}. In \cite{Kolar2000} I. Kol\'ar
showed that this is the behaviour of $M^{A_l}\to M^{A_{l-1}}$. In this paper \emph{we characterize
natural transformations that are affine bundles.} It is done easily by adopting a different point of view
of the tangent space of $M^A$ as done in \cite{Munoz2000}. Our result is as follows: \emph{there is
a canonical affine structure for natural transformations $M^A\to M^B$ induced by a surjective 
morphisms $A\to B$ whose kernel has null square.}
This is true for $M^{A_l}\to M^{A_k}$ with $2k+1\geq l > k \geq 0$.

In some cases the natural transformations induce maps between Jet spaces. This situation holds
in the cases studied in \cite{Kolar2000}. We characterize this situation, and moreover, 
\emph{we will determine when an affine structure on the Weil bundle morphism is passed down to
the Jet space morphism}. In addition to that, we will prove that in this case there also exist
an affine structure in the morphism between the groups of automorphisms of the Weil algebras. This is true for
spaces $J^l_mM \to J^k_m$ with $l>k>0$ and $3k+1\geq 2l$.

\subsection*{Acknowledgements}

I acknowledge Prof. J. Mu\~noz and R. J. Alonso-Blanco for their help in this work. 
I am grateful to J. Rodr\'iguez and S. Jim\'enez for their suggestions and helpful
discussions on the topic of Weil bundles. I would like also to thank professor C. Sim\'o
who hosted me in Barcelona University, and J. J. Morales-Ruiz for their support. 

\subsection*{Notation and Conventions}

All manifolds and maps are assumed to be infinitely differentiable. All results
involving a manifold $M$ assume that it is not empty and all results involving 
jet spaces $J^AM$ assume that the jet space $J^AM$ is also not empty (in such case,
maybe the algebraic conditions for the existence of affine structure are satisfied,
but no structure exists). 

\section{Weil Bundles}

By a \emph{Weil algebra} we mean a finite dimensional, local, commutative $\mathbb R$-algebra
with unit. If $A$ is a Weil algebra, let us denote by $\mathfrak m_A$ its maximal ideal. If $A$ and 
$B$ are Weil algebras, by a morphism $A\to B$ we mean a $\mathbb R$-algebra morphism.

\begin{example}
  Let $\mathbb R[[\xi_1,\ldots,\xi_m]]$ be the ring of formal series with real coefficients and
free variables $\xi_1,\ldots,\xi_m$. Let $\mathfrak m$ be the maximal ideal spanned by $\xi_1,\ldots,
\xi_m$. Then, for any non-negative integer $l$, the ring
$$\mathbb R^l_m = \mathbb R[[\xi_1,\ldots,\xi_m]]/\mathfrak m^{l+1}$$
is a Weil algebra.
\end{example}

For every Weil algebra $A$, there is a non-negative integer $l$ such that $\mathfrak m_A\neq 0$ but
$\mathfrak m_A=0$; we say that $l$ is the \emph{height} of $A$. The \emph{width} of $A$ is the dimension
of the vector space $\mathfrak m_A/\mathfrak m_A^2$. Thus, $\mathbb R^l_m$ is a Weil algebra of
height $l$ and width $m$. If $A$ is of height $l$ and width $m$ there exists a surjective morphism $\mathbb R^l_m\to A$ 
(see \cite{Kolar2000}, \cite{Munoz2000}).
 
\begin{definition}
  Let $M$ be a smooth manifold and $A$ a Weil algebra. The set $M^A$ of the $\mathbb R$-algebra morphisms
$$p^A\colon C^\infty(M)\to A,$$
is the so-called space of near-points of type $A$ of $M$, also called $A$-points of $M$.
\end{definition}

Let us consider a basis $\{a_k\}$ of $A$. For each $f\in C^\infty(M)$ we define real valued
functions $\{f_k\}$ on $M^A$ by setting:
$$p^A(f) = \sum_k f_k(p^A)a_k.$$
We say that the $\{f_k\}$ are the real components of $f$ relative to the basis
$\{a_k\}$. 

\begin{theorem}[\cite{Munoz2000}]
The space $M^A$ is endowed with a unique structure of a smooth manifold such that the real
components of smooth functions on $M$ are smooth functions on $M^A$. 
\end{theorem}

\begin{example} It is well known that each morphism $C^\infty(M)\to \mathbb R$ is a point
of $M$. Since the real components on $M^{\mathbb R}$ of smooth functions coincide with the
functions themselves we know that $M^{\mathbb R} = M$.
\end{example}

\begin{example}
For each Weil algebra $\mathbb R^l:m$ let us denote by $M^l_m$ the space of near points
of type $\mathbb R^l_m$. Then $M^1_1$ is the tangent bundle $TM$. In general $M^l_m$ is the space
of germs at the origin of smooth maps $\mathbb R^m\to M$ up to order $l$ (see \cite{Weil1953}). 
\end{example}

A Weil algebra morphism $\phi\colon A\to B$ induces by composition a smooth map 
$\hat\phi\colon M^A \to M^B$ \cite{Kolar1993, Munoz2000} which is called a \emph{natural
transformation,} $\hat\phi(p^A) = \phi\circ p^A$. 

\begin{example}
Let us notice that a Weil algebra is provided with a unique morphism $A\to\mathbb R$. It induces
a canonical map $M^A\to M$ which is a fibre bundle. This bundle is the so-called \emph{Weil bundle}
of type $A$ over $M$. Let $p^A$ be in $M^A$ and $p$ its projection in $M$. Then we will say that
$p^A$ is an $A$-point near $p$. For each smooth function $f$ the value $p^A(f)$ depends only on the germ
at $p$ of $f$.
\end{example}

A smooth map $f\colon M\to N$ of smooth manifolds induces by composition a $\mathbb R$-algebra morphism
$$f^*\colon C^\infty(N)\to C^\infty(M), \quad f^*(g) = g\circ f.$$
We can conpose this morphism $f^*$ with $A$-points of $M$ obtaining $A$-points of $N$. Then, Weil 
algebra morphisms and smooth maps transform near-points by composition, and this implies a functorial
behaviour of Weil bundles with respect to those transformations.

\medskip
 
We can formalize this situation in the following way. Let us consider $\mathcal M$ the category of smooth
manifolds, and $\mathcal W$ the category of Weil algebras. It the direct product category $\mathcal M\times
\mathcal W$, objects are pairs $(M,A)$ and morphisms are pairs $(f,\phi)$. We define $w(M,A) = M^A$. Thus, the
natural way of defining the natural image of the morphism $(f,\phi)$,
$$f\colon M\to N, \quad\quad \phi\colon A\to B,$$
is
$$w(f,\phi)\colon M^A\to N^B, \quad p^A\mapsto \phi\circ p^A\circ f^*.$$
The following result follows easily:

\begin{proposition}\label{Proposition1}
The assignment
$$w\colon \mathcal M\times\mathcal W \leadsto \mathcal M, \quad \quad (M,A)\leadsto M^A,$$
is a covariant functor.
\end{proposition}

\begin{remark}\label{Remark1}
There are two remarkable cases of induced maps:
\begin{itemize}
\item If $X\subset M$ is an embedding then for each Weil algebra $A$ the induced map $X^A\to M^A$ is
also an embedding.
\item If $A\to B$ is a surjective morphism then for all $M$ the induced natural transformation
$M^A\to M^B$ is a fibre bundle. 
\end{itemize}
\end{remark}

\begin{example}
Let $A$ be a of height $l$. For each $k\leq l$ let us define $A_k = A/\mathfrak m_A^{k+1}$. Then $A_k$ is 
a Weil algebra of height $k$ and $A_l= A$. For $k\geq r$ we have a natural projection $M^{A_k}\to M^{A_r}$
which is a bundle. In particular we have canonical bundles $M^k_m\to M^r_m$.
\end{example}

\subsection{Tangent Structure}

Given $p^A\in M^A$, let us denote by $\Der_{p^A}(C^\infty(M),A)$ the space of \emph{derivations} of the
ring $C^\infty(M)$ into the module $A$, where the structure of $C^\infty(M)$-module of $A$ is induced
by the point $p^A$ itself. By derivations we mean $\mathbb R$-linear maps $\delta\colon C^\infty(M)\to A$
satisfying Leibniz's formula:
\begin{equation}\label{Leibniz}
\delta(f\cdot g) = p^A(f)\cdot \delta(g) + p^A(g)\cdot\delta(f).
\end{equation}

If $D$ is a tangent vector to $M^A$ at $p^A$ then it defines a derivation give by:
$$\delta(f) = \sum_k (Df_k)a_k,$$
and it is easy to prove that the spaces $\Der_{p^A}(C^\infty(M),A)$ and $T_{p^A}(M^A)$ are
identified in this way \cite{Munoz2000}. Fron now on we will assume this identification. It 
applies not just to the vector spaces, but it is also compatible 
with Proposition \ref{Proposition1}. The following
thorem resume some results in \cite{Munoz2000}.

\begin{theorem}[Mu\~noz, Rodr\'iguez, Muriel \cite{Munoz2000}]\label{Theorem2}
Let us consider a smooth map $f\colon M\to N$, a Weil algebra morphism $\phi\colon A\to B$, and
the induced smooth map:
$$w(f,\phi)\colon M^A\to N^B, \quad p^A\to q^B = \phi\circ p^A\circ f^*.$$
Then, the linearized map $w(f,\phi)'\colon T_{p^A}(M^A)\to T_{q^B}(N^B)$ coincides (under the identification
assumed above) with the map:
$$\Der_{p^A}(C^\infty(M),A)\to \Der_{q^B}(C^\infty(N),B), \quad\quad \delta\mapsto \phi\circ \delta \circ f^*.$$
\end{theorem}

\subsection{Affine Structure}

In this section we will analyze the structure of the fibre bundle induced
by a surjective morphism $A\to B$ which has been introduced in Remark \ref{Remark1}.
In some specific cases it has been proved that those bundles are endowed with a canonical
structure of affine bundles. We will see that this structure has its foundation in the 
algebraic construction of the \emph{spaces of near-points}. Indeed, it is an easy task
to give an algebraic characterization of this fact. A morphism will induce an affine
structure if and only if its kernel ideal has null square.

The key point is to consider both, \emph{near-points} and \emph{tangent vectors} to $M^A$
as $\mathbb R$-linear maps from $C^\infty(M)$ to a Weil algebra. Thus, they are provided
with the addition law of $\mathbb R$-linear maps. Under some adequate assumptions we will obtain a new
near-point when we add a derivation to a near-point.

\begin{lemma}\label{Lemma1}
Let us consider $p^A\in M^A$ and $D\in T_{p^A}(M^A)$. The sum $p^A+D$ is
an $A$-point of $M$ if an only if $(\mathrm{Im}(D))^2=0$. 
\end{lemma}

\begin{proof}
Let us define $\tau = p^A + D$. Then, for all $f,g\in C^\infty(M)$, 
$$\tau(f\cdot g) = \tau(f)\cdot \tau(g) - D(f)\cdot D(g);$$
since it is a $\mathbb R$-linear map, it is an algebraic morphism if an only if
for any pair $f$ and $g$ of smooth functions in $M$ we have $D(f)\cdot D(g) = 0$. 
\end{proof}

\begin{lemma}\label{Lemma2}
Let us consider $p^A$ and $q^A$ in $M^A$. The difference $\delta = q^A - p^A$ is
a derivation and belongs to $T_{p^A}(M^A)$ if and only if $(\mathrm{Im}(\delta))^2 = 0$. 
\end{lemma}

\begin{proof}
  For all $f$ and $g$ in $C^\infty(M)$ we have
$$\delta(f\cdot g) = p^A(f)\cdot\delta(g)+p^A(g)\cdot\delta(f) + \delta(f)\cdot\delta(g).$$
Thus, $\delta$ satisfies Leibniz's formula if and only if for all $f$ and $g$ in $C^\infty(M)$ the
product $\delta(f)\cdot\delta(g)$ vanishes. 
\end{proof}

\begin{lemma}\label{Lemma3}
Let $I$ be an ideal of a $k$-algebra $A$ with $k$ a field of characteristic different from $2$. 
The square $I^2$ of the ideal vanishes if an only if for all $x\in I$ its square $x^2$ also vanishes. 
\end{lemma}

\begin{proof}
If $I^2$ vanishes it is clear that $x^2=0$ for all $x\in I$. Conversely, let us assume that
the square of all elements of $I$ vanish. Let $x$ and $y$ be in $I$, then:
$$0 = (x+y)^2 = x^2 + y^2 + 2xy = 2xy.$$
Hence, $xy = 0$ and $I^2$ vanish. 
\end{proof}

Let $\phi\colon A\to B$ be a surjective morphism of Weil algebras, and let $I$ be its kernel ideal. 
Let us consider a smooth manifold $M$, and the induced fibre bundle $\hat\phi\colon M^A\to M^B$. The
linearization $\hat\phi'$ gives rise to an exact sequence
$$0\to TV_{p^A}^{\hat\phi}(M^A) \to T_{p^A}(M^A)\to T_{p^B}(M^B) \to 0$$
which defines the vertical tangent sub-bundle $TV^{\hat\phi}(M^A)\subset T(M^A)$. Taking into account
that tangent vectors are derivations from $C^\infty(M)$ to $A$, we will notice that $D\in T_{p^A}(M^A)$
belongs to $TV_{p^A}^{\hat\phi}(M^A)$ if and only if $Im(D)\subseteq I$. Then $TV_{p^A}^{\hat\phi}(M^A)$ is
the space of derivations from $C^\infty(M)$ to $I$, where the structure of $C^\infty(M)$-module
in $I$ is given by the morphism $p^A\colon C^\infty(M)\to A$. 

Let us assume taht $I^2$ vanishes. Let us consider $p^A$ and $q^A$ in the same fibre of the bundle,
\emph{id est} $\hat\phi(p^A) = \hat\phi(q^A) = p^B$. Thus, $p^A$ and $q^A$ induce the same structure
of $C^\infty(M)$-module in $I$. Hence, the space of derivations $TV_{p^A}^{\hat\phi}(M^A)$ is canonically
ismorphic to $TV_{q^A}^{\hat\phi}(M^A)$. \emph{We will denote this space by $TV^{\hat\phi}_{p^B}(M^B)$.}

Using Lemma \ref{Lemma1} and \ref{Lemma2} we conclude that for any pair of $A$-points $p^A$ and $q^A$ in
the fibre of $p^B$ as above the difference $p^A-q^A$ is
a derivation which belongs to the space $TV_{p^B}^{\hat\phi}(M^B)$. And also, for any derivation 
$D\in TV_{p^B}^{\hat\phi}(M^B)$, the sum $p^A+D$ is a near-point of type $A$ in fiber of $p^B$. 
Thus, the natural law of addition of linear maps,
$$\hat\phi^{-1}(p^B)\times TV_{p^B}^{\hat\phi}(M^B) \to \hat\phi^{-1}(p^B),\quad (p^A,D)\mapsto p^A+D$$
induces an affine structure on the fibre $\hat\phi^{-1}(p^B)$ associated with the vector
space $TV^{\hat\phi}_{p^B}(M^B)$ of derivations from $C^\infty(M)$ to $I$. 

We define the vector bundle $TV^{\hat{\phi}}(M^B)$ on $M^B$, 
$$TV^{\hat\phi}(M^B)\to M^B,$$
whose fibre over a $B$-point $p^B$ is the space $TV^{\hat\phi}_{p^B}(M^B)$. Hence, $TV^{\hat\phi}(M^B)$ is the vector
bundle associated with the affine bundle $\hat\phi\colon M^A\to M^B$,
$$M^A\times_{M^B} TV^{\hat\phi}(M^B) \to M^A, \quad (p^A,D)\mapsto p^A + D.$$

On the other hand, if $I^2$ does not vanish, by applying Lemma \ref{Lemma3} we find a derivation
$D\colon C^\infty \to I$ such that $(\mathrm{Im}(D))^2$ does not vanish. Hence, $p^A+D$ does not belong
to $M^A$. We have proved the following:

\begin{theorem}\label{Theorem3}
Let $\phi\colon A\to B$ be a surjective Weil algebra morphism, and let $I$ be its kernel ideal. For any
manifold $M$, the natural addition law of linear maps induces an structure of affine bundle in the
fibre bundle $\hat\phi\colon M^A\to M^B$ if and only if $I^2=0$. 
\end{theorem}

  Be means of some elementary computations on the algebras $A_k$ and $\mathbb R^l_m$ we 
deduce the following corollaries to the Theorem \ref{Theorem3}.

\begin{corollary}\label{Corollary1}
Let $A$ be a Weil algebra of height $l$. The natural projection $M^A\to M^{A_k}$ is endowed
with a canonical structure of affine bundle if and only if $2k+1\geq l$. 
\end{corollary}

\begin{corollary}\label{Corollary2}
The natural projection of spaces of frames, $M^l_m\to M^k_m$ is endowed with a canonical structure
of affine bundle if and only if $2k+1\geq l$.
\end{corollary}

\section{Jet Spaces}

\begin{definition}\label{Definition2}
A jet of $M$ is an ideal of differentiable functions $\mathfrak p\subset C^\infty(M)$ such that
the quotient algebra $A_{\mathfrak p} = C^\infty(M)/\mathfrak p$ is a Weil algebra. A jet $\mathfrak p$ is
said to be of type $A$, or an $A$-jet, if $A_{\mathfrak p}$ is isomorphic to $A$. The set $J^AM$ of $A$-jets of
$M$ is the so calle $A$-jet space of $M$. 
\end{definition}

An a point $p^A$ of $M$ is said to be \emph{regular} if it is a surjective morphism. The set of regular
$A$-points of $M$ is denoted by $\check M^A$. It is a dense open subset of $M^A$. It is obvious that
an $A$-point is regular if and only if its kernel is an $A$-jet. Thus, we have a surjective map:
\begin{equation}\label{Equation2}
\ker\colon \check M^A \to J^A M.
\end{equation}

Let us consider $\Aut(A)$, the group of \emph{automorphisms} of $A$. It is a linear algebraic 
group, as can be seen easily by representing it as a subgroup of $GL(\mathfrak m_A)$ 
 (see \cite{Blazquez2005}). This group $\Aut(A)$ acts on $\check M^A$ by composition. Two 
$A$-points related by an automorphism have the same kernel ideal. Moreover, two $A$-points with the
same kernel ideal are related by an automorphism. In this way $J^AM$ is identified with the space
of orbits $\check M^A/\Aut(A)$, and its manifold structure is determined in this way (see \cite{Alonso2000}).

\begin{example}\label{Example6}
The group $G^l_m$ of automorphisms of $\mathbb R^l_m$ is called the \emph{$l$-th prolongation
of the linear group of rank $m$} (see \cite{Olver1999}), also called \emph{jet group}. In particular
$G^1_m$ is the linear group of rank $m$. The group $G^l_m$ is the group of transformations of $\mathbb R^l$
around a fixed point up to order $l$. 
\end{example}

\begin{theorem}[Alonso-Blanco \cite{Alonso2000}] \label{Theorem4} There is a unique structure of smooth manifold
on $J^AM$ such that the map $\ker$ (appearing in equation \eqref{Equation2}) is a principal bundle
with structural group $\Aut(A)$.
\end{theorem}

\begin{example}\label{Exmaple7}
Let us denote by $J^l_mM$ the space of jets of type $\mathbb R^l_m$ of $M$. Thus, $J^l_mM$ is the
space of germs of $m$-submanifolds of $M$ up to order $l$.
\end{example}

The space $J^A M$ is a bundle over $M$. We will say that that $\mathfrak p$ is a jet over the point $p$ 
if $\mathfrak p\subset \mathfrak m_p$, where $\mathfrak m_p$ is the ideal of smooth functions 
vanishing at $p$. If $p^A$ is an $A$-point \emph{near} $p$ then $\ker(p^A)$ is a jet \emph{over} $p$. 

\subsection{Functorial behaviour}

  In contrast with Weil bundles, jet spaces do not show
a functorial behaviour. A smooth map $f\colon M\to N$ induces a smooth map
on jet spaces, but in the general case it is defined only on an open dense subset
of $J^AM$, which depends on $f$. There is no natural object associated to a Weil
algebra morphism $A\to B$. There is a natural, highly interesting, object associated
to a pair $(A,B)$ of Weil algebras: the \emph{Lie correspondence}. This is a
submanifold $\Lambda_{A,B} M$ of the fibred product $J^AM\times_M J^BM$,
$$\Lambda_{A,B} M = \{ ( \mathfrak p, \bar{\mathfrak p}) \in J^AM\times_M J^BM \colon
\mathfrak p \subset \bar{\mathfrak p} \}.$$
The Lie correspondence is empty if and only if there does not exist any surjective
morphism from $A$ to $B$. There is a special case to be analyzed in which it is the
graph of a bundle $J^A M \to J^B M$.

Let $I$ be an ideal of $A$. Then, for each automorphism $\sigma$ of $A$, the space
$\sigma(I)$ is another ideal of $A$; the group $\Aut(A)$ acts in the set of ideals
of $A$. We say that $I$ is an \emph{invariant ideal} of $A$ if for all $\sigma\in\Aut(A)$
we have $I = \sigma(I)$. For each positive integer $k$ the $k$-th power of the
maximal ideal $\mathfrak m_A^k$ is an invariant ideal, and any other ideals obtained
from these by general processes of division and derivation are also invariant; some
examples are shown in \cite{Alonso2004}. Let $I\subset A$ be an invariant ideal and 
$\phi\colon A\to A/I = B$ the canonical projection into the quotient algebra $B$.
Let $p^A$ be an $A$-point and $\mathfrak p$ be its kernel. It is obvious that 
that the kernel ideal $\bar{\mathfrak p}$ of the composition $\phi\circ p^A$ is the unique $B$-jet
containing $\mathfrak p$. Let us denote by $\check \phi$ the restriction of $\hat\phi$ to 
the space of regular points $\check M^A$. We have a commutative diagram:
\begin{equation}\label{Equation3}
\xymatrix{\check  M^A \ar[r]^-{\check \phi} \ar[d] & \check M^B \ar[d] \\ J^A M \ar[r]^-{\phi^J} & J^BM }
\quad\xymatrix{p^A\ar[r]\ar[d] & \check \phi(p^A) = p^B \ar[d]\\ 
\mathfrak p\ar[r] & \bar{\mathfrak p}}
\end{equation}
The Lie correspondence is precisely the set
$$\Lambda_{A,B}M = \{(\mathfrak p, \phi^J(\mathfrak p)) \colon \mathfrak p \in J^AM\}.$$
Summarizing, the following result holds:

\begin{theorem}\label{Theorem5}
If $I\subset A$ is an invariant ideal and $B$ is the quotient algebra $A/I$ then there is
a canonical bundle structure $J^AM\to J^BM$. 
\end{theorem}

\subsection{Tangent structure}

In order to study the linearization of $\phi^J$ in diagram \eqref{Equation3} we need some
characterization of the tangent space to $J^AM$ at a jet $\mathfrak p$.

\begin{theorem}[\cite{Alonso2000, Munoz2000}] \label{Theorem6}
 The space $T_{\mathfrak p}(J^AM)$ realizes itself
canonically as a quotient of the space of derivations $C^\infty(M)\to A_{\mathfrak p}$. A derivation
$\delta$ defines the null vector if and only if $\delta(\mathfrak p) = 0$. Thus,
$$T_{\mathfrak p}(J^AM) \simeq \Der(C^\infty(M),A_{\mathfrak p})/\Der(A_{\mathfrak p},A_{\mathfrak p}),$$
\end{theorem}

  In order to a better understanding let us give some sketch of the proof. Let us remind that the Lie
algebra of $\Aut(A)$ is the space of derivations $\Der(A,A)$ as can be shown in a matrix representation
of the group (see \cite{Munoz2000, Blazquez2005}). Taking $p^A\in \check M^A$ such that 
$\ker(p^A) = \mathfrak p$, the representation of $\Der(A,A)$ as fundamental vector fields of the action
of $\Aut(A)$ on $\check M^A$ gives rise to an exact sequence:
$$0 \to \Der(A, A) \xrightarrow{\mbox{\footnotesize{fun. vec. fields}}} 
T_{p^A}(M^A) \xrightarrow{\ker'} T_{\mathfrak p}J^A(M) \to 0.$$

  Let us take into account that $T_{p^A}(M^A) $ is the space of derivations \linebreak $\Der_{p^A}(C^\infty(M), A)$ and
that $p^A$ induces a isomorphism of $C^\infty$-algebras between $A$ and $A_{\mathfrak p}$. Thus,
it follows the isomorphism of the theorem. This isomorphism does not depends on the $A$-point $p^A$
representing the $A$-jet $\mathfrak p$. It can be seen by means of the principal structure stated
in Theorem \ref{Theorem4}. 

\section{Affine Structure on Jet Spaces} 

\subsection{Space of Regular Points}

  Let $I$ be an ideal of the Weil algebra $A$, and $\phi\colon A \to B$ the canonical projection
onto the quotient algebra $B=A/I$.

\begin{lemma}[\cite{Munoz2000}]\label{Lemma4} 
A finite set $\{a_1,\ldots,a_m\}\subset \mathfrak m_A$ is a system of
generators of $A$ if and only if the set of their classes $\{\bar a_1,\ldots \bar a_m\}$ in $\mathfrak m_A/
\mathfrak m_A^2$ is a basis of $\mathfrak m_A/
\mathfrak m_A^2$.
\end{lemma}

\begin{lemma}\label{Lemma5}
If $I \nsubseteq \mathfrak m_A^2$ then there exists a non trivial subalgebra $S\subset A$ such that
$S/(S\cap I)\simeq B$. 
\end{lemma}

\begin{proof}
If $I \nsubseteq \mathfrak m_A^2$ then the canonical projection 
$\mathfrak m_A/\mathfrak m_A^2 \to \mathfrak m_B/\mathfrak m_B^2$ has non-trivial kernel. 
There exists a finite set $\{a_1,\ldots,a_r\}$ such that 
$\{\overline{\phi(a_1)}, \ldots, \overline{\phi(a_r)}\}$ is a basis of $\mathfrak m_B/\mathfrak m_B^2$
but $\{\bar a_1,\ldots, \bar a_r\}$ is not a basis of $\mathfrak m_A/\mathfrak m_A^2$. Then
$S = \mathbb R[a_1,\ldots, a_r]$ is a proper subalgebra and verifies $S/(S\cap I) = B$. 
\end{proof}

  Note that each subalgebra of a $A$ is a Weil algebra. For each subset $X\subset \mathfrak m_A$,
$\mathbb R[X]$ is a Weil algebra and  its maximal ideal is spanned by $X$. 

\begin{lemma}\label{Lemma6}
  The following conditions are equivalent:
\begin{itemize}
\item[(1)] $I\subseteq \mathfrak m_A^2$. 
\item[(2)] $\hat\phi^{-1}(\check M^B) \subseteq \check M^A$. 
\end{itemize}
\end{lemma}

\begin{proof}
  Let us assume $I\subset \mathfrak m_A^2$, and consider $p^A\in M^A$ such that
$\hat\phi(p^A)$ is a regular $B$-point. There are differentiable functions $f_1,\ldots f_m$
in $M$ such that $\{\phi(p^A(f_1)),\ldots \phi(p^A(f_m))\}$ is a system of generators
of $B$. Then their classes modulo $\mathfrak m_B^2$ form a basis 
$\{\overline{\phi(p^A(f_1))},\ldots \overline{\phi(p^A(f_m))}\}$ of $\mathfrak m_B/\mathfrak m_B^2$.
Since $I$ is contained in $\mathfrak m_A^2$ and $\mathfrak m_B = \mathfrak m_A/I$ we have that
$\mathfrak m_B/\mathfrak m_B^2 \simeq \mathfrak m_A/\mathfrak m_A^2$. Then $\{\overline{p^A(f_1)},
\ldots,\overline{p^A(f_m)}\}$ is a basis of $\mathfrak m_A/\mathfrak m_A^2$ and
$\{p^A(f_1)\ldots,p^A(f_m)\}$ is a system of generators of $A$. Thus, $A$ is regular. 

Conversely, let us assume $I\nsubseteq \mathfrak m_A^2$. Consider a subalgebra $S\subset A$ as in
Lemma \ref{Lemma5}. Then $M^S\to M^B$ is a bundle. Let $p^B\in\check M^B$ be a regular $B$-point, and $p^S$
any preimage of $p^B$. Hence, $p^S$ is a $S$-point, and thus a non regular $A$-point, but $\hat\phi(p^S) = p^B$. 
\end{proof}

From now on we will consider the anihilator ideal of $I$, 
$$\Ann(I) = \{a\in A\colon \forall b \in I \,\, ab=0 \}.$$
Let us notice that $I\subseteq \Ann(I)$ if and only if $I^2 = 0$.

\begin{theorem}\label{Theorem7}
The bundle $\check\phi\colon \check M^A \to \check M^B$ is endowed with a
canonical structure of affine bundle (given by the addition law of morphisms and
derivations) if and only if $I\subseteq \mathfrak m_A^2 \cap \Ann(I)$. 
\end{theorem}

\begin{proof}
Suppose that $I\subseteq \mathfrak m_A^2\cap \Ann(I)$. Then the addition law of
$A$-points and derivations induces an affine structure on $\hat \phi$. Let $p^B\in \check M^B$
be a regular $B$-point. In view of Lemma \ref{Lemma6} the fibre $\hat\phi^{-1}(p^B)$ consist of
regular points. Thus, $\check\phi^{-1}(p^B) = \hat\phi^{-1}(p^B)$ so that the bundle $\check M^A\to \check M^B$
is the restriction of $M^A\to M^B$ to the open submanifold $\check M^B$; which is an affine bundle.  

On the other hand, let us assume that $I\nsubseteq (\mathfrak m_A^2\cap\Ann(I))$. If $I\nsubseteq \Ann(I)$,
then the addition of an $A$-points and a derivation is not in general an $A$-opint and there is no
affine structure. Finally, let us assume that $I\subseteq \Ann(I)$ but $I\nsubseteq \mathfrak m_A^2$.
Then there is an affine structure on $\hat\phi$. However, by Lemma \ref{Lemma6} there is a non-regular
$A$-point $p^A M$ such that its projection $p^B$ is regular. Let us consider $q^A\in\check\phi^{-1}(p^B)$,
and $D = p^A-q^A\in TV_{q^A}^{\check \phi}\check M^A$. Thus $q^A+D\in \check M^A$, and there is not
affine structure.  
\end{proof}

\begin{corollary}\label{Corollary3}
Let $A$ be of height $l$. 
Then for each $l>k>0$ the natural projection $\check M^A\to \check M^{A_k}$ is an affine
bundle if and only if $2k+1\geq l$.
\end{corollary}

\begin{corollary}\label{Corollary4}
For any $l>k>0$, the natural projection $\check M^l_m \to \check M^k_m$ is an affine
bundle if and only if $2k+1\geq l$.
\end{corollary}

\subsection{Affine Structure on the Group of Automorphisms}

Let $\subset A$ be an invariant ideal of the Weil algebra $A$, and $\phi\colon A \to B$
the canonical projection into the quotient algebra. Each automorphism $\sigma\in \Aut(A)$
verifies $\sigma(I)=I$, thus it induces an automorphism $\phi_*(\sigma)\in \Aut(B)$. 

\begin{definition}\label{Definition3}
We will call the affine sequence associated to $I$ the following sequence of 
algebraic groups:
$$K(I) \to \Aut(A) \xrightarrow{\phi_*} \Aut(B), $$
where
$$K(I) = \{\sigma\in\Aut(A)\,\colon\,\sigma(a)-a\in I,\, \sigma(b) = b,\,\forall a\in A\forall b\in I\},$$
is the subgroup of automorphisms of $A$ inducing the identity both in $B$ and $I$. 
\end{definition} 

We will say that the affine sequence is \emph{exact on the left side} in $K(I) = \ker{\phi_*}$.
Analogously, we will say that it is \emph{exact on the right side} if $\phi_*$ is
surjective. Note that if it is exact both on the right and left sides then it is an
exact sequence.

Let us notice that if $I\subseteq \Ann(I)$ then the $A$-module $I$ is also a $B$-module.
By composition we have a canonical inmersion $\Der(B, I)\subseteq \Der(A,I)$ identifiying
derivations from $B$ to $I$ with derivations from $A$ to $I$ which vanish on $I\subset A$. 

\begin{proposition}\label{Proposition2}
Let us asume $I\subseteq \Ann(I)$. The affine sequence associated to $I$ is exact on the
left side if and only if $\Der(B,I) = \Der(A,I)$. 
\end{proposition}

\begin{proof}
  Assuming that the affine sequence is exact on the left side, let us consider the
sequence of Lie algebras induced by the sequence of algebraic groups associated to $I$. 
The Lie algebra of $K(I)$ is, by the definition of $K(I)$, the space of derivations
from $A$ to $I$ which vanish on $I$. Thus, it is identified with the space $\Der(B, I)$. On
the other hand the kernel of the Lie algebra morphism induced by $\phi_*$ is the space
$\Der(A,I)$. If the affine sequence is exact on the left side, then the Lie algebra of
$K(I)$ coincides with this last one, and $\Der(B,I)=\Der(B,A)$. 

Conversely, let us assume that $\Der(A,I) = \Der(B,I)$, \emph{i.e.} all
derivations from $A$ to $I$ vanish on $I$. Let $\sigma$ be an automorphism
of $A$. The difference $Id_A-\sigma$ is a derivation from $A$ to $I$. It vanishes
on $I$, and thus for any $a\in I$ we have $\sigma(a) = a$, and then $\sigma$
induces the identity in $I$, $\sigma\in K(I)$.
\end{proof}

\begin{theorem}\label{Theorem8}
  If $I \subseteq \Ann(A)\cap \mathfrak m_A^2$ and the affine sequence is exact, then
$\phi_*$ is endowed with a natural structure of affine bundle associated with the space
$\Der(A,I)$ with the following addition law:
$$\sigma \oplus D = \sigma + \sigma\circ D.$$
\end{theorem}

\begin{proof}
Let $D$ be a derivation from $A$ to $I$. Then, $Id_A+D$ is an automorphism of $A$. Conversely,
let $\sigma$ be an automorphism of $A$ such that $\phi_*(\sigma) = Id_B$. In such case
$\sigma-Id_A$ is a derivation and it takes values in $I$. We have:
$$\Der(A,I) = \ker(\phi_*).$$
By definition of the addition law we have that $\sigma\oplus D = \sigma(Id+ D)$, so that 
$\sigma\oplus \Der(A,I) = \sigma\circ\ker(\phi_*)$. Finally let us see that the addition
law of the bundle is compatible with the vector space structure of $\Der(A,I)$:
$$(\sigma \oplus D)\oplus D' = \sigma \oplus (D+D').$$
From Proposition \ref{Proposition2} we have that,
$$(\sigma\oplus D\oplus D' = \sigma + \sigma \circ D  (\sigma + \sigma\circ D)\circ D' =
\sigma \oplus(D+D') + \sigma \circ D \circ D'.$$
And because of each derivation vanish on $I$ we have that $D\circ D$ vanish.
\end{proof}

\begin{lemma}\label{Lemma7}
 If $I\subseteq Ann(I)^2$ then the affine sequence associated to $I$ is exact on the
left side. 
\end{lemma}

\begin{proof}
  Let us consider a derivation $D\colon A \to I$, and $a$ in $I$; thus $a$ is also in
in $\Ann(I)^2$ and the we can write $a = \sum b_kc_k$ for suitable $b_k$ and $c_k$
in $Ann(I)$. We have,
$$D(a) = \sum_k b_kD(c_k) + c_kD(b_k) = 0.$$
Thus, $D$ anhililates $I$. We conclude that $\Der(A,I) = \Der(B,I)$. Our assertion
follows directly from Proposition \ref{Proposition2}
\end{proof}

\begin{corollary}\label{Corollary5}
If the natural numbers $l>r>0$ verify $3r+1\geq 2l$ then the natural projection
$G^l_m\to G^r_m$ is an affine bundle.
\end{corollary}

\begin{proof} 
In general $G^l_m\to G^k_m$ is a surjective morphism. We apply the Lemma \ref{Lemma7}
to the case $A=\mathbb R^l_m$, $I=\mathfrak m_A^{k+1}$. Then, $\Ann(I)=\mathfrak m_A^{l-k}$
and $\mathfrak m_A^{k+1}\subseteq \Ann(\mathfrak m_A^{k+1})^2$ if and only if $k+1\geq 2(l-k)$. 
\end{proof}

\subsection{Affine structure on Jet bundles}

Let $I\subset A$ be an invariant ideal with $I\subseteq \Ann(I)\cap \mathfrak m_A^2$ and 
let us denote by $B$ the quotient algebra $A/I$ as above. 
For each $\mathfrak p\in J^AM$ let us denote by $\pi_{\mathfrak p}\colon C^\infty(M)\to A_{\mathfrak p}$ the
canonical projection and $\bar{\mathfrak p} = \phi^J(\mathfrak p)\in J^BM$. Then $A_{\mathfrak p}\simeq \mathfrak p$
and $\bar{\mathfrak p}/\mathfrak p \simeq I$. For each $D\in \Der(C^\infty(M), \bar{\mathfrak p}/\mathfrak p)$ let
us define,
\begin{equation}\label{Equation4}
\mathfrak p + D = \ker(\pi_{\mathfrak p } + D).
\end{equation}
In such case, because $I\subseteq\Ann(I)$ we have that 
$\pi_{\mathfrak p} + D$ is an $A_{\mathfrak p}$-point. It is regular
because $I\subseteq\mathfrak m_A^2$. Hence, $\mathfrak p + D$ is an
$A$-jet. We also have that  $\phi^J(\mathfrak p + D) = \bar{\mathfrak p}$,
because $D$ takes values in $\bar{\mathfrak p}/\mathfrak p$.

\begin{lemma}\label{Lemma8}
Each derivation $D\colon C^\infty(M)\to\bar{\mathfrak p}/\mathfrak p$ which vanishes on $\mathfrak p$
also vanishes on $\bar{\mathfrak p}$ if and only if the affine sequence associated to $I$ is exact
on the left side.
\end{lemma}

\begin{proof} A derivation $C^\infty(M)\to \bar{\mathfrak p}/\mathfrak p$ which anihilates $\mathfrak p$
factorizes through a derivation $A_{\mathfrak p}\to\bar{\mathfrak p}/\mathfrak p$. Then, the claim
es equivalent to Lemma \ref{Lemma2}.
\end{proof} 
  
\begin{theorem}
The addition law \eqref{Equation4} defines an affine structure on the bundle $\phi^J\colon J^AM\to J^BM$ for
any smooth manifold $M$ if and only if the affine squence associated to $I$ is exact.
\end{theorem}

\begin{proof}
A derivation $D\colon C^\infty(M)\colon\bar{\mathfrak p}/\mathfrak p$ defines a tangent vector
$[D]\in T_{\mathfrak p}(J^AM)$ as it is shown in Theorem \ref{Theorem6}. 
Moreover, we have that $[D]\in TV_{\mathfrak p}^{\phi^J}J^AM$, because $D$ takes values in $\bar{\mathfrak p}/\mathfrak p$. Let us prove that the following conditions, which are equivalent to the assertion of the theorem, 
hold if and only if the affine sequence associated to $I$ is exact. 
\begin{enumerate}
\item If two derivations $D$ and $D'$ from $C^\infty(M)$ to $A_{\mathfrak p}$ define the same
tangent vector at $\mathfrak p$ then $\mathfrak p+D = \mathfrak p+ D'$. 
\item The natural projection $\Der(C^\infty(M),\bar{\mathfrak p}/\mathfrak p)\to TV^{\phi^J}_{\mathfrak p}(J^AM)$
is surjective.
\item For each $\mathfrak q \subset \bar{\mathfrak p}$ there is a unique $[D]\in TV_{\mathfrak p}^{\phi^J}(J^AM)$
such that $\mathfrak p + [D] = \mathfrak q$. 
\item For each $A$-jet $\mathfrak q$ contained in the $B$-jet $\bar{\mathfrak p}$ there is a canonical
isomorphism $TV^{\phi^J}_{\mathfrak p}(J^AM) \simeq TV^{\phi^J}_{\mathfrak q}(J^AM)$.
\end{enumerate}

\medskip

\emph{Condition} (1.) \emph{holds if and only if the affine sequence is exact on the left side}.

\medskip

Let $D$ and $D'$ define the same tangent vector $[D]\in TV_{\mathfrak p}^{\phi^J}(J^AM)$. In such case
the difference $\delta = D - D'$ vanishes on $\mathfrak p$. By Lemma \ref{Lemma8}, each derivation
vanishing on $\mathfrak p$ also vanish on $\bar{\mathfrak p}$ if and only if the affine sequence associated
to $I$ is exact on the left side. In the case of exact affine sequence we have:
$$\ker(\pi_{\mathfrak p} + D) = \ker(\pi_{\mathfrak p} + D').$$

\medskip

\emph{If the affine sequence is exact on the right side then condition} (2.) \emph{holds.}

\medskip

It is an application of the classical \emph{Snake Lemma}. We have a natural diagram of
exact columns and arrows:
$$\xymatrix{ 
& 0 \ar[d] & 0 \ar[d] & 0 \ar[d] \\
0 \ar[r] & \Der(A_{\mathfrak p},\bar{\mathfrak p}/\mathfrak p) \ar[r]\ar[d] & 
    \Der(A_{\mathfrak p},A_{\mathfrak p})  \ar[r]^{\psi} \ar[d] & 
    \Der(A_{\bar{\mathfrak p}},A_{\bar{\mathfrak p}}) \ar[d] \\
0 \ar[r] & \Der(C^\infty(M),\bar{\mathfrak p}/\mathfrak p) \ar[r] \ar[d]^-{\bar\psi} & 
    \Der(C^\infty(M),A_{\mathfrak p}) \ar[r]\ar[d] & 
    \Der(C^\infty(M),A_{\bar{\mathfrak p}})\ar[r]\ar[d] & 0  \\
0 \ar[r] & TV_{\mathfrak p}^{\phi^J}(J^AM) \ar[r] & T_{\mathfrak p}(J^AM)\ar[r]\ar[d] & 
    T_{\bar{\mathfrak p}}(J^BM)  \ar[r] \ar[d] & 0 \\
& & 0 & 0
}$$

According to the Snake Lemma, if $\mathrm{coker}(\psi)$ vanishes then we have an exact sequence
$$\ldots \to 0 \to \mathrm{coker}(\bar\psi) \to 0 \to \ldots $$
and vice-versa. Hence, $\mathrm{coker}(\psi)$ vanish if and only if $\mathrm{coker}(\bar\psi)$ vanishes. 
Note that the natural mapping $\psi$ is the linearization of the algebraic group morphism 
$\Aut(A_{\mathfrak p}) \to \Aut(A_{\bar{\mathfrak p}})$. Since $A_{\mathfrak p} \simeq A$ and 
$A_{\bar{\mathfrak p}}\simeq B$ we conclude that if the affine sequence associated to $I$ is
exact on the right side then (2.) holds.

\medskip

\emph{Condition (3.) holds if and only if the affine sequence associated to $I$ is exact on the right side}

\medskip

Let us consider any other  $A$-jet $\mathfrak q\subset \bar{\mathfrak p}$ and an isomorphism
$\tau\colon A_{\mathfrak q}\to A_{\mathfrak p}$. Thus, we have diagram (not commutative):
$$\xymatrix{C^\infty(M) \ar[r]^-{\pi_{\mathfrak q}}\ar[rd]_-{\pi_{\mathfrak q}} & 
 A_{\mathfrak p} \ar[rd]^-{\bar\pi_{\mathfrak p}} \\
& A_{\mathfrak q} \ar[r]_{\bar\pi_{\mathfrak q}} \ar[u]^-{\tau} & A_{\bar{\mathfrak p}} }$$ 
 
Let us prove the following assertion: 
for each $\mathfrak q$ as above we can find an isomorphism $\tau$ such that 
$\bar\pi_{\mathfrak p}\circ \tau = \bar\pi_{\mathfrak q}$ if and only if the affine
sequence associated to $I$ is exact on the right side. First, let us assume that
the affine sequence is exact on the right side. Let us consider $p^A$ and $q^A$
two $A$-jet representing $\mathfrak p$ and $\mathfrak q$ respectively. The $B$-points,
$\check\phi(p^A)$ and $\check\phi(q^A)$ represent the same $B$-jet $\bar{\mathfrak p}$. Then, 
$\check\phi(p^A)$ and $\check\phi(q^A)$ are related by an automorphism of $B$, $\tau_1$. If the 
affine sequence associated to $I$ is exact on the right side, then $\Aut(B)$ is a quotient of
$\Aut(A)$. Hence, $\tau_1$ lifts to an automorphism $\tau_2$ of $A$. This automorphism
induces the isomorphism $\tau$ when we replace $A$ for $A_{\mathfrak p}$ and $A_{\mathfrak q}$.
Conversely, if the affine sequence is not exact on the right side, we can choose $\mathfrak q$ and
$A$-points $p^A$ and $q^A$ such that $\check\phi(p^A)$ and $\check\phi(p^B)$ are related by an automorphism
which can not be lifted to $A$. In such case, we can not find such isomorphism $\tau$.

Now, let us consider that the affine sequence is exact on the right side and let $\tau$ be as above.
Then $\pi_{\mathfrak p}$ and $\tau\circ\pi_{\mathfrak p}$ are regular $A_{\mathfrak p}$-points
that are projected onto the same $A_{\bar{\mathfrak p}}$-point. Then, $D = \pi_{\mathfrak p} - \tau\circ\pi_{\tau_q}$
is a derivation of $C^\infty(M)$ and it take values in $\bar{\mathfrak p}/\mathfrak p$. It defines
a vertical vector $[D]\in TV_{\mathfrak p}^{\phi^J}(J^AM)$ and it follows that:
$$\mathfrak p + [D] = \mathfrak q.$$

\medskip

\emph{If the affine sequence is exact then condition} (4.) \emph{holds}

\medskip

If the affine sequence is exact, we can find $\tau$ and $\bar\tau\colon A_{\mathfrak q}\to A_{\mathfrak p}$
as above. Then, $\sigma = \tau\circ\bar\tau$ is an automorphism of $A_{\mathfrak p}$ which induces the 
identity on $A_{\bar{\mathfrak p}}$. Since the affine sequence is exact on the right side we have
that $\sigma$ induces the identity map on $\bar{\mathfrak p}/\mathfrak p$. it follows that 
the restriction of $\tau$ to the space $\bar{\mathfrak p}/\mathfrak q$ is canonical and does not
depend on $\tau$. This canonical identification 
$\tau\colon \bar{\mathfrak p}/\mathfrak q \to \bar{\mathfrak p}/\mathfrak p$ induces canonical isomorphisms:
$$\xymatrix{\Der(A_{\bar{\mathfrak p}},\bar{\mathfrak p}/\mathfrak q) \ar[r]\ar[d] & 
 \Der(C^\infty(M),\bar{\mathfrak p}/\mathfrak q) \ar[r]\ar[d] & TV_{\mathfrak q}^{\phi^J}(J^AM)\ar[d]^-{\tau^*} \\ 
 \Der(A_{\bar{\mathfrak p}},\bar{\mathfrak p}/\mathfrak p) \ar[r] & 
 \Der(C^\infty(M),\bar{\mathfrak p}/\mathfrak p)  \ar[r] & TV_{\mathfrak p}^{\phi^J}(J^AM)
}$$ 
Thus, condition (4.) is satisfied. 

\medskip

If the affine sequence is exact then the vector space $TV_{\mathfrak p}^{\phi^J}(J^AM)$ depends
only on the base $B$-jet $\bar p$. Those spaces define a vector bundle $TV^{\phi^J}(J^BM)\to J^BM$
and the composition law:
$$J^AM \times_{J^BM}TV^{\phi^J}(J^BM) \to J^AM, \quad (\mathfrak p + [D])\mapsto \mathfrak p + [D].$$
is an affine structure on the bundle $\phi^J$.
\end{proof}

\begin{corollary}\label{Corollary6}
Let $A_l$ be of height $l$, and $l>k>0$. The natural projection $J^{A_l}M \to J^{A_k}M$ is endowed with
a canonical structure of affine bundle if and only if $3k+1\geq 2l$ and $\Aut(A_l)\to\Aut(A_k)$ is
surjective.
\end{corollary}

\begin{corollary}
The natural projection $J^l_mM\to J^r_mM$ fir $l>r>0$ is endowed with a canonical structure of
affine bundle if and only if if $3r+1\geq 2l$.
\end{corollary}

\begin{remark}
Those results extend the well known affine structure of the spaces of jets of sections. First,
they show that this structure arises not only for the projection by lower order one-by-one, but if follows
an arithmetic formula which is also different from the expected one of \emph{duplicating orders}. Second,
this affine structure is inherent to the spaces $J^l_mM$ as spaces of ideals, it does not depend
on their realization as spaces of sections of fibre bundles.
\end{remark}

\bigskip

{\sc\noindent David Bl\'azquez-Sanz \\
Escuela de Matem\'aticas\\
Universidad Sergio Arboleda\\
Calle 74, no. 14-14 \\
Bogot\'a, Colombia\\
}
E-mail: {\tt david.blazquez-sanz@usa.edu.co}

\end{document}